\documentclass{article}
\usepackage[cp1251]{inputenc}
\usepackage[russian]{babel}
\usepackage{amssymb,amsfonts,amsmath,amsthm,amscd}
\usepackage{graphicx}
\usepackage{enumerate}
\usepackage{soul}
\usepackage[matrix,arrow,curve]{xy}
\usepackage{longtable}
\usepackage{array}
\usepackage{bm}

\newdimen\symskip
\newdimen\defskip
\newdimen\parind
\parind=\parindent
\newdimen\leftmarge
\newdimen\theoremshape
\topsep\defskip
\makeatletter
\newcommand*{\clei}{\nobreak\hskip\z@skip}

\renewcommand{\:}{\textup{:}}
\renewcommand{\~}{\textup{;}}
\DeclareRobustCommand*{\ti}{~\textemdash{} }
\DeclareRobustCommand*{\dh}{\clei\hbox{-}\clei}
\newcommand{\no}{}
\renewcommand{\@listI}{\settowidth\labelwidth{\labheadi{\no}}\listipar{\parind}{\labelwidth}}
\newcommand{\listivpar}{\topsep\defskip\partopsep0pt\parsep-\parskip\itemsep0.5\topsep}
\newcommand{\listipar}[2]{\rightmargin0pt\leftmargin#1\labelsep#1\advance\labelsep-#2\itemindent0pt\listivpar}
\renewcommand{\@listii}{\settowidth\labelwidth{\labheadii{\@roman{\no}}}\listiipar{\parind}{\labelwidth}}
\newcommand{\listiivpar}{\topsep0.5\defskip\partopsep0pt\parsep-\parskip\itemsep0.5\topsep}
\newcommand{\listiipar}[2]{\rightmargin0pt\leftmargin#1\labelsep#1\advance\labelsep-#2\itemindent0pt\listiivpar}
\def\thempfn{\ifcase\value{footnote}1\or *\or **\or ***\else\@ctrerr\fi}
\renewcommand\footnoterule{%
  \kern-3\p@
  \hrule\@width1in
  \kern2.6\p@}
\makeatother

\makeatletter

\renewcommand{\@biblabel}[1]{[#1]}
\renewenvironment{thebibliography}[1]
     {\renewcommand{\refname}{References}%
      \section*{\refname}%
      \@mkboth{\MakeUppercase\refname}{\MakeUppercase\refname}%
      \list{\@biblabel{\@arabic\c@enumiv}}%
           {\itemsep\baselineskip
            \leftmargin\parind
            \settowidth\labelwidth{\@biblabel{#1}}%
            \labelsep\parind\advance\labelsep-\labelwidth
            \@openbib@code
            \usecounter{enumiv}%
            \let\p@enumiv\@empty
            \renewcommand\theenumiv{\@arabic\c@enumiv}}%
      \sloppy
      \clubpenalty4000
      \@clubpenalty\clubpenalty
      \widowpenalty4000%
      \sfcode`\.\@m}
     {\def\@noitemerr
       {\@latex@warning{Empty `thebibliography' environment}}%
      \endlist}

\def\@maketitle{%
  \newpage
  \vskip0.5em%
  UDK \udk%
  \vskip0.5em%
  MSC \msc%
  \vskip1em%
  \begin{center}\bf%
  \let\footnote\thanks%
   {\Large\@author\par}%
   \vskip1.5em%
   {\LARGE\@title\par}%
   \vskip1em%
   {\large\@date}%
  \end{center}%
  \par
  \vskip1.5em}

\def\@title{\@latex@warning@no@line{No \noexpand\title given}}

\makeatother

\makeatletter

\renewcommand\sectionmark[1]{%
 \markright{%
  \ifnum \c@secnumdepth >\z@
   \thesection. \ %
  \fi
 #1}}%

\renewcommand{\section}{\@startsection{section}{1}{0pt}%
{5.5ex plus .5ex minus .2ex}{1.5ex plus .3ex}%
{\center\normalfont\Large\bfseries\sffamily\bom}}
\renewcommand{\subsection}{\@startsection{subsection}{2}{0pt}%
{4.5ex plus .4ex minus .2ex}{0.75ex plus .2ex}%
{\center\normalfont\large\bfseries\sffamily\bom}}
\renewcommand{\subsubsection}{\@startsection{subsubsection}{3}{0pt}%
{2.5ex plus .5ex minus .2ex}{1ex plus .2ex}%
{\center\normalfont\bfseries\sffamily\bom}}

\newcommand{\Ss}{\textup{\S\,}}

\def\@postskip@{\hskip.5em\relax}

\def\postsection{.\@postskip@}

\def\postsubsection{.\@postskip@}

\def\postsubsubsection{.\@postskip@}

\def\postparagraph{.\@postskip@}

\def\postsubparagraph{.\@postskip@}

\def\@seccntformat#1{\csname pre#1\endcsname\csname the#1\endcsname\csname post#1\endcsname}

\makeatother

\renewcommand{\thesection}{\textup{\arabic{section}}}

\newcommand{\parr}{\par\addvspace{\defskip}}
\newcommand{\theo}[2]{\newtheorem{#1}{#2}[section]}
\newcommand{\deff}[2]{\newenvironment{#1}{\parr\textbf{#2.}}{\parr}}
\theo{cas}{Case}
\theo{problem}{Problem}
\theo{theorem}{Theorem}
\theo{lemma}{Lemma}
\theo{prop}{Proposition}
\theo{stm}{Statement}
\theo{imp}{Corollary}
\theo{ex}{Example}
\deff{df}{Definition}
\deff{note}{Remark}
\deff{denote}{Notation}
\deff{denotes}{Notations}
\deff{hint}{Hint}
\deff{answer}{Answer\:}

\makeatletter
\def\@begintheorem#1#2[#3]{%
  \deferred@thm@head{\the\thm@headfont \thm@indent
    \@ifempty{#1}{\let\thmname\@gobble}{\let\thmname\@iden}%
    \@ifempty{#2}{\let\thmnumber\@gobble}{\let\thmnumber\@iden}%
    \@ifempty{#3}{\let\thmnote\@gobble}{\let\thmnote\@iden}%
    \thm@notefont{\bfseries\upshape}%
    \indent%
    \thm@swap\swappedhead\thmhead{#1}{#2}{#3}%
    \the\thm@headpunct
    \thmheadnl 
    \hskip\thm@headsep
  }%
  \ignorespaces}
\renewenvironment{proof}{\setcounter{cas}{0}\parr\pushQED{\qed}\normalfont$\square\quad$}{\setcounter{cas}{0}\popQED\@endpefalse\parr}

\makeatother

\newcommand{\labheadi}[1]{\textup{#1)}}
\newcommand{\labheadii}[1]{\textup{(#1)}}

\newenvironment{nums}[1]{\renewcommand{\no}{#1}\begin{enumerate}}{\end{enumerate}}
\newcommand{\eqn}[1]{\begin{equation}#1\end{equation}}
\newcommand{\equ}[1]{\begin{equation*}#1\end{equation*}}

\newcommand{\case}[1]{\begin{cases}#1\end{cases}}

\newcommand{\rbmat}[1]{\begin{pmatrix}#1\end{pmatrix}}

\makeatletter

\def\LT@makecaption#1#2#3{%
  \LT@mcol\LT@cols c{\hbox to\z@{\hss\parbox[t]\LTcapwidth{%
    \sbox\@tempboxa{#1{#2. }#3}%
    \ifdim\wd\@tempboxa>\hsize
      #1{#2. }#3%
    \else
      \hbox to\hsize{\hfil\box\@tempboxa\hfil}%
    \fi
    \endgraf\vskip\baselineskip}%
  \hss}}}

\@addtoreset{equation}{section}
\@addtoreset{footnote}{section}

\newenvironment{casks}{%
  \matrix@check\casks\env@casks
}{%
  \endarray\right.%
}
\def\env@casks{%
  \let\@ifnextchar\new@ifnextchar
  \left\lbrack
  \def\arraystretch{1.2}%
  \array{@{}l@{\quad}l@{}}%
}

\newcounter{numt}
\newcounter{col}
\newcounter{coll}

\makeatother

\renewcommand{\ge}{\geqslant}
\renewcommand{\le}{\leqslant}
\newcommand{\fa}{\,\forall\,}

\newcommand{\bes}{\infty}
\newcommand{\es}{\varnothing}

\newcommand{\subs}{\subset}
\newcommand{\sups}{\supset}

\newcommand{\sm}{\setminus}
\newcommand{\wo}{\backslash}
\newcommand{\cln}{\colon}
\newcommand{\nl}{\lhd}

\newcommand{\Ra}{\Rightarrow}
\newcommand{\Lra}{\Leftrightarrow}

\newcommand{\dv}{\smash{\mskip3mu\lower1pt\hbox{\vdots}\mskip3mu}}

\newcommand{\ol}{\overline}

\newcommand{\wt}{\widetilde}

\newcommand{\sst}[1]{\substack{#1}}

\newcommand{\suml}[2]{\sum\limits_{{#1}}^{{#2}}}
\newcommand{\sums}[1]{\sum\limits_{{#1}}}

\newcommand*{\bw}[1]{#1\nobreak\discretionary{}{\hbox{$\mathsurround=0pt #1$}}{}}
\newcommand{\sco}{,\ldots,}

\newcommand{\und}[1]{\underbrace{#1}}

\newcommand{\ha}[1]{\left\langle#1\right\rangle}

\newcommand{\br}[1]{\bigl(#1\bigr)}
\newcommand{\Br}[1]{\Bigl(#1\Bigr)}

\newcommand{\ter}[1]{\textup{(}#1\textup{)}}

\newcommand{\bs}[1]{\bigl[#1\bigr]}

\newcommand{\bse}[1]{\bigl\lceil#1\bigr\rceil}

\newcommand{\bc}[1]{\bigl\{#1\bigr\}}
\newcommand{\BC}[1]{\Bigl\{#1\Bigr\}}

\newcommand{\mbb}{\mathbb}

\newcommand{\mcl}{\mathcal}

\newcommand{\N}{\mbb{N}}

\newcommand{\F}{\mbb{F}}

\newcommand{\Mc}{\mcl{M}}
\newcommand{\Pc}{\mcl{P}}

\newcommand{\ep}{\varepsilon}
\newcommand{\la}{\lambda}

\newcommand{\Sig}{\Sigma}

\newcommand{\ph}{\varphi}

\DeclareMathOperator{\Ker}{Ker}

\DeclareMathOperator{\End}{End}

\DeclareMathOperator{\rk}{rk}

\DeclareMathOperator{\cha}{char}
\DeclareMathOperator{\dep}{def}

\newcommand{\bom}{\boldmath}

\begin{document}

\author{O.\,G.\?Styrt}
\title{On matrix sets\\
invariant under conjugation\\
and taking linear combinations\\
of commuting elements}
\date{}
\newcommand{\udk}{512.643.1+512.643.721+512.643.8+512.643.875}
\newcommand{\msc}{15A20+15A27+47A05+47A08}

\maketitle

{\leftskip\parind\rightskip\parind

Subsets of a~matrix algebra over a~field that are invariant under conjugation and contain the linear span of each two of their commuting elements are
described. They obviously include the subsets of diagonalizable and nilpotent matrices. In the paper, the case of an algebraically closed field is considered.
The problem is easily reduced to description of subsets of diagonalizable matrices and subsets of nilpotent matrices with the given properties. So, among
diagonalizable matrices, there are four of such subsets. As for the nilpotent case, it is proved that the subset should be defined by the condition that the sizes
of all Jordan cells of the matrix belong to a~certain number set. An explicit criterion is obtained in terms of this set.

\smallskip

\textbf{Key words\:} matrix algebra, diagonalizable matrix, nilpotent matrix.\par}

\section{Introduction}\label{introd}

In researching abstract theory of Lie groups and algebras\ti for example, in \cite{VO,Hum}\ti different forms of \textit{Jordan decomposition} play quite important
role. So, semisimple Lie algebras over algebraically closed fields of characteristic zero involve two special types of elements\: \textit{semisimple} and
\textit{nilpotent} ones\ti and the uniquely defined \textit{additive Jordan decomposition} of an arbitrary element into the sum of two commuting elements of these
types. Each of the subsets of all semisimple and of all nilpotent elements is preserved under the inner automorphisms and contains the linear span of each two of its
commuting elements. The proof of this fact is based on elementary concepts of linear algebra and does not use the \textit{Lie} structure of the \textit{associative}
operator algebra (it is worth noting that, similarly, two given types of elements of semisimple Lie algebras are defined through the analogical properties of the
corresponding adjoint operators). At the same time, these two subsets fundamentally differ intersecting only by the zero element. It led to the problem of describing
all subsets with the same properties of invariance. In the paper, it is solved for the operator algebra (also involving two above-mentioned types of elements) that
gives perspectives for considering semisimple Lie algebras. Presumably, the case of a~simple Lie algebra of type~$A_r$ can be reduced to the result obtained by
concerning the (not needed in this paper) Lie structure of the operator algebra.

So, let $\F$ be a~field, $n$ a~positive integer, and $A$ the algebra $\End(\F^n)$ (naturally identified with $M_n(\F)$). The paper is devoted to describing subsets
$M\subs A$ such that
\begin{align}
&&0\in M&;\label{ze}\\
&\fa g\in A^*\quad& gMg^{-1}=M&;\label{conj}\\
&\fa X,Y\in M\quad& \br{[X,Y]=0}&\quad\Ra\quad\br{\ha{X,Y}\subs M}.\label{comm}
\end{align}
The family~$\Mc$ of all these sets is obviously closed under intersection. Besides, it includes the subset~$A_s$ of all diagonalizable operators and the
subset~$A_n$ of all nilpotent operators.

An arbitrary subset of the algebra~$A$ all of whose elements pairwise commute will be called \textit{commutative}.

\begin{note} The condition~\eqref{comm} is equivalent to the fact that the intersection of~$M$ with any commutative subspace in~$A$ is a~subspace. Also,
if it holds, then we have $\F M\subs M$ (it suffices to consider equal $X$ and~$Y$) and, thus, the condition~\eqref{ze} becomes equivalent to
non-emptiness of~$M$.
\end{note}

Further, let us suppose that $\F=\ol{\F}$ (consequently, $|\F|=\bes$). In the algebra~$A$, denote the subspace of all matrices with trace~$0$ by~$A_0$ and the semisimple and nilpotent
parts of an arbitrary element~$X$ by $X_s$ and~$X_n$ respectively.

\begin{lemma}\label{proj} If $M\in\Mc$, then, for each $X\in M$, we have $X_s,X_n\in M$.
\end{lemma}

\begin{proof} Assume that $X$ is a~Jordan matrix (it is true up to conjugation). Then $(X_s)_{ij}=0$ once $j\ne i$ and $(X_n)_{ij}=0$ once $j-i\ne1$.
Hence, for an arbitrary element $a\in\F^*\sm\{1\}$ and the diagonal matrix $g\in A^*$, $g_{kk}:=a^{-k}$, we have $gX_sg^{-1}=X_s$ and $gX_ng^{-1}=aX_n$
implying $Y:=X_s+aX_n=gXg^{-1}\in M$. The subspace $\ha{X,Y}=\ha{X_s,X_n}\subs A$ is commutative and, thus, lies in~$M$.
\end{proof}

Let $\Mc_s$ (resp.~$\Mc_n$) be the family of all $M\in\Mc$ contained in~$A_s$ (resp. in~$A_n$).

\begin{theorem} One has mutually inverse bijections
\equ{\begin{aligned}
\Mc\to\Mc_s\times\Mc_n,&\quad\quad M\to(M\cap A_s,M\cap A_n);\\
\Mc_s\times\Mc_n\to\Mc,&\quad\quad (M_s,M_n)\to\{X\in A\cln X_s\in M_s,\,X_n\in M_n\}.
\end{aligned}}
\end{theorem}

\begin{proof} In the algebra~$A$, the following holds\: if elements $X$ and~$Y$ commute, then the subset $\{X,X_{s,n},Y,Y_{s,n}\}$ is commutative, any
$a,b\in\F$ satisfying $(aX+bY)_{s,n}=aX_{s,n}+bY_{s,n}$. It remains to apply Lemma~\ref{proj}.
\end{proof}

Let $M$ be an arbitrary subset of the family~$\Mc_s$. Its intersection~$L$ with the commutative subspace $A_d\subs A$ of all diagonal matrices is
a~subspace, and, also, by~\eqref{conj},
\eqn{\label{diag}
M=\bc{gXg^{-1}\cln X\in L,\,g\in A^*}.}
Hence, the subspace $L\subs A_d$ is invariant under permutations of diagonal elements and, therefore, coincides with one of the subspaces
$0,\,\F E,\,A_d,\,A_d\cap A_0$. Applying~\eqref{diag} again, we obtain that $M$ is one of the subsets $\{0\},\,\F E,\,A_s,\,A_s\cap A_0$. Conversely,
each of these subsets obviously belongs to~$\Mc_s$.

Further, let us describe subsets of the family~$\Mc_n$.

The conjugacy class of each nilpotent matrix is uniquely defined by the non-ordered set of the cell sizes of its Jordan form (with considering
multiplicities). Therefore, any subset of the family~$\Mc_n$ is defined among~$A_n$ by the condition that this set belongs to a~certain family of sets.
Moreover, we will further show that each subset of the family~$\Mc_n$ is the set $M(Q)$ ($Q\subs\{2\sco n\}$) of all nilpotent matrices, in whose Jordan
forms, all cells have sizes from $Q\cup\{1\}$. Thus, the problem of describing the family~$\Mc_n$ is reduced to that of searching all subsets
$Q\subs\{2\sco n\}$ such that the set $M:=M(Q)$ satisfies~\eqref{comm} (the conditions \eqref{ze} and~\eqref{conj} obviously hold).

Once $Q=\es$, we have $M(Q)=\{0\}\in\Mc_n$. As for the case of non-empty~$Q$, the following criterion of the inclusion $M(Q)\in\Mc_n$ will be obtained in
the article.

\begin{theorem}\label{main} Assume that $Q\ne\es$. Then $M(Q)\in\Mc_n$ if and only if there exists a~number
$m_0\in\BC{2\sco\bs{\frac{n}{2}}+1}$ equal either to $\bs{\frac{n}{2}}+1$ or to $(\cha\F)^k$ \ter{$k\in\N$} and such that the set~$Q$ contains all numbers
$2\sco m_0$ all its rest elements being at most $2m_0$ and at least $n-m_0+2$, i.\,e.
\begin{gather}
2\sco m_0\in Q;\label{fir}\\
\fa m\in Q\quad(m_0+1\le m\le n)\quad\Ra\quad(n-m_0+2\le m\le2m_0).\label{oth}
\end{gather}
\end{theorem}

In the case $m_0=\bs{\frac{n}{2}}+1$ the condition~\eqref{oth} automatically holds\: $\frac{n}{2}<\bs{\frac{n}{2}}+1=m_0$, $n<2m_0$, $n\le2m_0-1$,
$m_0+1=(2m_0-1)-m_0+2\ge n-m_0+2$. Therefore, the main result immediately implies the following theorem.

\begin{theorem}\label{submain} The condition $2\sco\bs{\frac{n}{2}}+1\in Q$ is sufficient for the inclusion $M(Q)\bw\in\Mc_n$ and, once $\cha\F=0$ and
$Q\ne\es$, equivalent to it.
\end{theorem}

\section{Notations and auxiliary facts}

Let $V$ be an $n$\dh dimensional space, $A$ the algebra $\End(V)$, and $A_n\subs A$ the subspace of all nilpotent operators. For $X\in A$, set
$\dep(X):=\dim(\Ker X)$.

Consider an arbitrary operator $X\in A_n$ and its Jordan form. Denote by $h(X)$ the maximal size of a~cell. It is known that, for each $p\in\N$, the number
of cells of sizes $\ge p$ is equal to $\dep(X^p)-\dep(X^{p-1})$ \ter{in part, the total number of cells equals $\dep(X)$}. Besides,
$h(X)=\min\{k\bw\in\N\cln X^k\bw=0\}\bw\in\{1\sco n\}$.

The algebra $\End(\F^m)$ ($m\in\N$) will be naturally identified with $M_m(\F)$. Denote by~$J_{\la,m}$ the Jordan cell of size~$m$ with the eigenvalue $\la\in\F$.

For an arbitrary $f\in\br{\F[t]}\wo\{0\}$, set $k(f):=\max\{k\ge0\cln f\dv t^k\}$.

\begin{lemma}\label{cent} The centralizer of the matrix~$J_{0,m}$ is the subalgebra $\F[J_{0,m}]$.
\end{lemma}

We omit the proof since it is well-known.

\begin{lemma}\label{def} For any $f\in\br{\F[t]}\wo\{0\}$, we have $\dep\br{f(J_{0,m})}=\min\bc{m;k(f)}$.
\end{lemma}

\begin{proof} Set $X:=f(J_{0,m})$ and $k:=k(f)$. Then,
\begin{gather*}
\fa i,j=1\sco m\\
\begin{aligned}
(j-i<k)\quad&\Ra\quad(x_{ij}=0);\\
(j-i=k)\quad&\Ra\quad(x_{ij}\ne0).
\end{aligned}
\end{gather*}
Hence, $\rk X = \max\{0;m-k\}$, $\dep(X)=m-(\rk X)=\min\{m;k\}$.
\end{proof}

\begin{lemma}\label{int} If $f\in\br{\F[t]}\wo\{0\}$, $k:=k(f)\in(0;m]$, $q:=\bs{\frac{m}{k}}$, and $r:=m-kq$, then the matrix $f(J_{0,m})$ is nilpotent, and its
Jordan form contains exactly $k$ cells, $r$ of them having size $q+1$ and all the rest\ti size~$q$.
\end{lemma}

\begin{proof} Set $X:=f(J_{0,m})$. According to condition,
\equ{\begin{aligned}
k>0&&\quad&\Ra\quad\quad f\dv t,\ \ f^m\dv t^m\quad\Ra\quad X^m=0;\\
k\le m&&\quad&\Ra\quad\quad q\ge1;\\
q\le\frac{m}{k}< q+1 &&\quad&\Ra\quad\quad kq\le m < k(q+1),\ \ 0\le r<k.
\end{aligned}}
By Lemma~\ref{def},
\equ{
\fa p\ge0\quad\quad\dep(X^p)=\min\bc{m;k(f^p)}=\min\{m;kp\}.}

Consider the Jordan form of the matrix~$X$. For any $p\in\N$, the number~$k_p$ of cells of size $\ge p$ equals $\dep(X^p)-\dep(X^{p-1})=\min\{m;kp\}-\min\bc{m;k(p-1)}$.
In part,
\equ{\begin{aligned}
\fa p\in[1;q]\quad\quad&k_p=\min\{m;kp\}-\min\bc{m;k(p-1)}=kp-k(p-1)=k;\\
&k_{q+1}=\min\bc{m;k(q+1)}-\min\{m;kq\}=m-kq=r;\\
&k_{q+2}=\min\bc{m;k(q+2)}-\min\bc{m;k(q+1)}=m-m=0.
\end{aligned}}
Note also that the total number of cells equals $k_1=k$.
\end{proof}

\begin{imp}\label{inp} In the conditions of the previous lemma, the set of sizes of all cells \ter{without considering multiplicities} of the Jordan form of the matrix
$f(J_{0,m})$ is $\BC{\bs{\frac{m}{k}};\bse{\frac{m}{k}}}$.
\end{imp}

\begin{proof} In the notations of Lemma~\ref{int}, the numbers of cells of sizes $q$ and $q+1$ equal respectively $k-r$ and~$r$, cells of other sizes not existing.
Besides, $0\le r<k$, $k-r>0$, $\bse{\frac{m}{k}}\in\{q;q+1\}$, and, also, $\Br{\bse{\frac{m}{k}}=q}\Lra\br{\frac{m}{k}=q}\Lra(r=0)$.
\end{proof}

\begin{imp}\label{poly} If $f\in\br{\F[t]}\wo\{0\}$ and $k(f)=1$, then the matrix $f(J_{0,m})$ is similar to~$J_{0,m}$.
\end{imp}

\begin{lemma}\label{bin} In the space~$V$, consider commuting nilpotent operators $X$ and~$Y$, and, also, $X$\dh invariant subspaces $W$ and~$U$, such that $V=W\oplus U$
and $YW\subs U$. Besides, assume that, in the Jordan form of the operator~$X|_W$, all cells have sizes greater than $k+1$ where $k:=h(X|_U)$. Then,
\equ{
\fa p,q\ge0\quad(p+q>\dim U)\quad\Ra\quad(Y^qX^pU=0).}
\end{lemma}

\begin{proof} It follows from condition that
\begin{gather*}
X^kU=0\ne X^{k-1}U;\\
\fa i=1\sco k\quad\quad n_i:=\dim\br{\Ker(X^i|_U)}-\dim\br{\Ker(X^{i-1}|_U)}>0;\\
X^kYW\subs X^kU=0,\quad X^kYW=0.
\end{gather*}

Let $i\in\{0\sco k\}$ be an arbitrary number. Denote by~$V_i$ the (obviously $Y$\dh invariant) subspace $\Ker(X^i)+(X^{k-i+1}V)\subs V$. According to condition,
\begin{gather}
\begin{aligned}
&\Ker(X^i)=\Ker(X^i|_W)\oplus\Ker(X^i|_U),&\quad\quad&X^{k-i+1}V=(X^{k-i+1}W)\oplus(X^{k-i+1}U);\\
&X^{k-i+1}U\subs\Ker(X^i|_U),&\quad\quad&X^{k-i+1}W\sups\Ker(X^{i+1}|_W);
\end{aligned}\notag\\
V_i=(X^{k-i+1}W)\oplus\Ker(X^i|_U)\sups\Ker(X^{i+1}|_W)\oplus\Ker(X^i|_U).\label{vi}
\end{gather}
Suppose additionally that $i\ne0$. Then,
\begin{gather}
XV_i\subs V_{i-1}\subs V_i;\notag\\
X^{i-1}(Y^{n_i}X^{k-i+1}W)=Y^{n_i-1}(X^kYW)=0,\notag\\
Y^{n_i}(X^{k-i+1}W)\subs\Ker(X^{i-1})\subs V_{i-1}.\label{dec1}
\end{gather}
Besides, $\Ker(X^i|_W)\oplus\Ker(X^{i-1}|_U)\subs V_{i-1}\cap\Ker(X^i)$,
\equ{
\dim\br{V_{i-1}\cap\Ker(X^i)}\ge\dim\br{\Ker(X^i|_W)}+\dim\br{\Ker(X^{i-1}|_U)}=\dim\br{\Ker(X^i)}-n_i.}
Thus, in the space~$V$, the subspaces $V_{i-1}$ and $\wt{V}_{i-1}:=V_{i-1}+\Ker(X^i)\sups V_{i-1}$ are invariant under the nilpotent operator~$Y$, and, also,
\equ{
\dim\br{\wt{V}_{i-1}/V_{i-1}}=\dim\br{\Ker(X^i)}-\dim\br{V_{i-1}\cap\Ker(X^i)}\le n_i,}
implying $Y^{n_i}\br{\wt{V}_{i-1}/V_{i-1}}=0$, $Y^{n_i}\br{\Ker(X^i)}\subs V_{i-1}$, $Y^{n_i}\br{\Ker(X^i|_U)}\subs V_{i-1}$. Therefore (see \eqref{vi}
and~\eqref{dec1}), $Y^{n_i}V_i\subs V_{i-1}$.

Applying~\eqref{vi} again, we obtain that
\begin{itemize}
\item $V_0=X^{k+1}W$, $YV_0=YX^{k+1}W=X(X^kYW)=0$\~
\item $V_k\sups\Ker(X^k|_U)=U$.
\end{itemize}
Note also that $\suml{i=1}{k}n_i=\dim\br{\Ker(X^k|_U)}=\dim U$.

Assume that there exist $p,q\ge0$ such that $p+q>\dim U$ and $Y^qX^pU\ne0$. Then, $X^pU\ne0$, and, thus, $p<k$, $0<k-p\le k$,
$X^pU\subs X^pV_k\subs V_{k-p}$. For $r:=\suml{i=1}{k-p}n_i$, we have
\begin{gather}
\dim U-r=\Br{\suml{i=1}{k}n_i}-\Br{\suml{i=1}{k-p}n_i}=\Br{\sums{k-p<i\le k}n_i}\ge p>\dim U-q,\notag\\
r<q,\quad q\ge r+1;\label{geq}\\
Y^rX^pU\subs Y^rV_{k-p}\subs V_0,\quad Y^{r+1}X^pU\subs YV_0=0.\label{br}
\end{gather}
By \eqref{geq} and~\eqref{br}, $Y^qX^pU=0$. So, we came to a~contradiction that proves the claim.
\end{proof}

\begin{imp}\label{triv} In the conditions of Lemma~\ref{bin}, each of the operators $Y^qX^p$ \ter{$p\ge0$, $q\ge1$, $p+q>\dim U+1$} is trivial.
\end{imp}

\begin{proof} We have $q':=q-1\ge0$ and $p+q'>\dim U$ implying $Y^{q'}X^pU=0$. Hence,
\equ{
Y^qX^pU=Y(Y^{q'}X^pU)=0,\quad\quad Y^qX^pW=(Y^{q'}X^p)(YW)\subs Y^{q'}X^pU=0.\qedhere}
\end{proof}

\begin{imp}\label{por} In the conditions of Lemma~\ref{bin}, the following holds\:
\equ{
\fa r\ge(\dim U+2)\ \ \fa Z\in A\quad\br{[Z,X]=[Z,Y]=0}\quad\Ra\quad\br{(Y+XZ)^r=(XZ)^r}.}
\end{imp}

\begin{proof} We have $(Y+XZ)^r=\sums{\sst{p,q\ge0;\\p+q=r}}\br{C_r^qY^q(XZ)^p}=(XZ)^r$ since, once $q\ge1$, the corresponding summand is $C_r^qY^qX^pZ^p=0$
(see Corollary~\ref{triv}).
\end{proof}

\begin{theorem}\label{comh} If $X,Y\in A_n$, $[X,Y]=0$, $m:=h(X)$, $l:=h(Y)$, $m\ge\frac{n}{2}+1$, and $l>n-m+2$, then $l=\bse{\frac{m}{p}}$ for some $p\in\N$.
\end{theorem}

\begin{proof} By condition, $2m\ge n+2$, $n-m\le m-2$. Further, in the space~$V$, there exists a~basis $(e_1\sco e_n)$ in which the operator~$X$ has
a~Jordan matrix with the first cell of size~$m$. So, $V$ is the direct sum of the $X$\dh invariant subspaces $W:=\ha{e_i\cln i\le m}$
and $U:=\ha{e_i\cln i>m}$. Note that $k:=h(X|_U)\le\dim U=n-m<m-1$ implying $X^{m-1}U=0$. For the subspaces $\wt{W}:=\ha{e_i\cln i<m}$
and $\wt{V}:=\Ker(X^{m-1})=\wt{W}\oplus U$, we have $Y\wt{V}\subs\wt{V}$, $Y\in A_n$, and $\dim\br{V/\wt{V}}=1$. Therefore, $YV\subs\wt{V}$,
$Ye_m\in YV\subs\wt{V}=\wt{W}\oplus U$, i.\,e. $Ye_m=w+u$ ($u\in U$, $w\in\wt{W}=\ha{X^ie_m\cln i\in\N}$), and, thus, $w=(XZ)e_m$ ($Z\in\F[X]$).

Note that $XZ=f(X)$ where $f\in\br{\F[t]}\wo\{0\}$ and $f\dv t$. So, $p:=k(f)\in\N$.

The subalgebra $\wt{A}:=\F[X,Y,Z]=\F[X,Y]\subs A$ is commutative. Hence, $\wt{A}_n:=\wt{A}\cap A_n\nl\wt{A}$. Also, $X,Y,Z\in\wt{A}$ and $X,Y\in\wt{A}_n$
implying $\wt{Y}:=Y-XZ\in\wt{A}_n$. Further,
\begin{gather}
\wt{Y}e_m=Ye_m-(XZ)e_m=Ye_m-w=u\in U;\notag\\
\wt{Y}W=\wt{Y}\Br{\br{\F[X]}e_m}=\br{\F[X]}\br{\wt{Y}e_m}\subs\br{\F[X]}U=U.\label{wu}
\end{gather}
Thus, $X,\wt{Y}\in\wt{A}_n$ are commuting nilpotent operators. Besides, $m-1>k$, $m>k+1$. Using Corollary~\ref{por} and the relation~\eqref{wu}, we obtain
for an arbitrary $r\ge\dim U+2$ the equality $Y^r=\br{\wt{Y}+XZ}^r=(XZ)^r=f^r(X)$ that, since $k(f^r)=pr$, implies the equivalences
\equ{\begin{aligned}
(r\ge l)\quad&\Lra\quad(Y^r=0)\quad&&\Lra\quad\br{f^r(X)=0}\quad&&\Lra\quad(f^r\dv t^m)\quad\Lra\\
&\Lra\quad\br{k(f^r)\ge m}\quad&&\Lra\quad(pr\ge m)\quad&&\Lra\quad\br{r\ge\frac{m}{p}}.
\end{aligned}}
By condition, $l,l-1\ge n-m+2=\dim U+2$. Hence, $l\ge\frac{m}{p}>l-1$, $l=\bse{\frac{m}{p}}$.
\end{proof}

\begin{imp}\label{hx} If $X,Y\in A_n$, $[X,Y]=0$, and $m:=h(X)\ge\frac{n}{2}+1$, then each cell size of the Jordan form of the operator~$Y$, that is greater
than $n-m+2$ and $\bse{\frac{m}{2}}$, equals~$m$.
\end{imp}

\begin{proof} By Theorem~\ref{comh}, at least one of the following conditions holds\:
\begin{nums}{-1}
\item\label{lenm} $h(Y)\le n-m+2$\~
\item\label{lem2} $h(Y)\le\bse{\frac{m}{2}}$\~
\item\label{eqm} $h(Y)=m$.
\end{nums}
In the case~\ref{lenm} (resp.~\ref{lem2}), all cells of the Jordan form of the operator~$Y$ have sizes $\le n-m+2$ (resp. $\le\bse{\frac{m}{2}}$).
In the case~\ref{eqm}, one of these cells is of size~$m$, and all the rest\ti of sizes $\le n-m$.
\end{proof}

\begin{lemma} If $m>\frac{n}{2}$, $X,Y\in A_n$, and $X^m=Y^m=[X,Y]=0$, then
\equ{
\fa p,q\ge0\quad(p+q=m)\quad\Ra\quad(Y^qX^p=0).}
\end{lemma}

\begin{proof} Use induction by~$n$. The case $n=1$ is obvious. So, prove the claim for $n>1$ with an assumption that it is already proved for all less values.

Since $n<2m$, the number of cells of sizes $\ge m$ in the Jordan form of the operator~$X$ is $\le1$ and, on the other hand, equal to
$\dep(X^m)-\dep(X^{m-1})=\rk(X^{m-1})-\rk(X^m)=\rk(X^{m-1})$. In the space~$V$, the subspace $X^{m-1}V$ of dimension $\rk(X^{m-1})\le1$ is invariant under
the nilpotent operator~$Y$, and, hence, $Y(X^{m-1}V)=0$, $YX^{m-1}=0$.

First, suppose that $\dep(Y)\ge2$.

The subspace $\wt{V}:=YV\subs V$ is $X$- and $Y$\dh invariant. Besides, $X^{m-1}\wt{V}\bw=X^{m-1}YV=0$, $Y^{m-1}\wt{V}=Y^mV=0$, and, also,
$\dim\wt{V}=n-\dep(Y)\le n-2<n,2(m\bw-1)$. Applying the inductive hypothesis to the commuting nilpotent operators
$X|_{\wt{V}},Y|_{\wt{V}}\bw\in\End\br{\wt{V}}$, we obtain
\equ{
\fa p,q\ge0\quad(p+q=m-1)\quad\Ra\quad\br{Y^qX^p\wt{V}=0}.}
For arbitrary $p,q\ge0$, such that $p+q=m$, we have $Y^qX^p=0$. Indeed, if $q=0$, then $p=m$ and $Y^qX^p=X^m=0$. If $q\ge1$, then
$q':=q-1\ge0$ and $p+q'=m-1$ implying $Y^{q'}X^p\wt{V}=0$, $Y^qX^pV=Y^{q'+1}X^pV=Y^{q'}X^pYV=Y^{q'}X^p\wt{V}=0$.

Now assume that $\dep(Y)=1$.

The Jordan form of the operator~$Y$ includes exactly one cell. According to Lemma~\ref{cent}, the operator~$X$ belongs $\F[Y]$ and, being nilpotent, has
form $YZ$ where $Z\in\F[Y]$. Clearly, $[Y,Z]=0$. Hence, if $p,q\ge0$ and $p+q=m$, then $Y^qX^p=Y^q(YZ)^p=Y^mZ^p=0$.
\end{proof}

\begin{imp}\label{mhal} If $m>\frac{n}{2}$, $X,Y\in A_n$, and $X^m=Y^m=[X,Y]=0$, then, for any $Z\in\ha{X,Y}$, we have $Z^m=0$.
\end{imp}

\begin{lemma}\label{mdeg} If $m,k\in\N$, $m=(\cha\F)^k$, $X,Y\in A_n$, and $X^m=Y^m=[X,Y]=0$, then, for any $Z\in\ha{X,Y}$, we have $Z^m=0$.
\end{lemma}

\begin{proof} According to condition, $p:=\cha\F$ is a~prime positive integer. On the commutative algebra $\wt{A}:=\F[X,Y]\subs A$, the mapping of raising to the
power~$p$ is an endomorphism of the \textit{ring}~$\wt{A}$\~ the same can be said about the mapping $\ph\cln\wt{A}\to\wt{A}$ of raising to the power~$m$.
Also, $\ha{X},\ha{Y}\subs\Ker\ph$ implying $\ha{X,Y}=\ha{X}+\ha{Y}\subs\Ker\ph$.
\end{proof}

\section{Proofs of the results}

Recall that the main field~$\F$ is supposed to be algebraically closed (and, consequently, infinite).

Consider the set~$\Pc$ of all finite non-ordered sets of non-unit positive integers, \textit{taking multiplicities into account} while, in part, defining
the \textit{order}~$|P|$, the \textit{element sum} $\Sig(P)$, the \textit{union} operation~$\cup$, and, also, while defining a~set by listing its elements
in figure brackets. Assigning to each matrix $X\in A_n$ the set $G(X)\in\Pc$ of all non-unit sizes of its Jordan cells, we obtain a~surjective
mapping~$G$ of the set~$A_n$ to the subset $\Pc(n):=\bc{P\in\Pc\cln\Sig(P)\le n}\subs\Pc$ whose fibres are exactly the orbits of the action of~$A^*$
on~$A_n$ by conjugations.

Let $M\in\Mc_n$ be an arbitrary subset.

By~\eqref{conj}, $M=\bc{X\in A_n\cln G(X)\in\Pc_M}$ ($\Pc_M\subs\Pc(n)$).

\begin{theorem} If $P_1,P_2,P:=P_1\cup P_2\in\Pc(n)$, then $(P\in\Pc_M)\Lra(P_1,P_2\in\Pc_M)$.
\end{theorem}

\begin{proof} Let $X\in A_n$ be an arbitrary matrix such that $G(X)=P$. It is naturally represented in the form $X_1+X_2$, where $X_i\in A_n$,
$G(X_i)=P_i$, $[X_1,X_2]=0$, $M':=\{X_1\bw+aX_2\cln a\bw\in\F^*\}\subs A_n$, and $G(M')=\{G(X)\}=\{P\}$ (see Corollary~\ref{poly}). Besides, the
subspace $L:=\ha{M'}=\ha{X_1,X_2}\subs A$ is commutative, and, hence, its subset $M\cap L$ is a~subspace. It rests to note that
\equ{\begin{aligned}
(P\in\Pc_M)&\quad\Lra\quad&(M'\subs M)&\quad\Lra\quad&(M'\subs M\cap L)&\quad\Lra\quad(M\cap L=L);\\
(P_1,P_2\in\Pc_M)&\quad\Lra\quad&(X_1,X_2\in M)&\quad\Lra\quad&(X_1,X_2\in M\cap L)&\quad\Lra\quad(M\cap L=L).\qedhere
\end{aligned}}
\end{proof}

\begin{imp}\label{sep} An arbitrary set $\{n_1\sco n_k\}\in\Pc(n)$ belongs to~$\Pc_M$ if and only if $\{n_i\}\in\Pc_M$ for any $i=1\sco k$.
\end{imp}

Thus, the subset~$M$ is uniquely defined by the subset~$Q_M$ of all $m\in\{2\sco n\}$ such that $\{m\}\in\Pc_M$. Namely, to each
subset $Q\subs\{2\sco n\}$, corresponds a~subset $M(Q)\subs A_n$ (see \Ss\ref{introd}), and then $M=M(Q_M)$. Obviously,
$(Q_M=\es)\Lra\br{M=\{0\}}$.

\begin{lemma}\label{pow} If $k\in\N$ and $X\in M$, then $X^k\in M$.
\end{lemma}

\begin{proof} Assume that $k>1$ (otherwise, it is nothing to prove).

The matrix $Y:=X+X^k$ is similar to~$X$ (see Corollary~\ref{poly}) and, therefore, lies in~$M$. Besides, $[X,Y]=0$ implying $M\sups\ha{X,Y}\ni Y-X=X^k$.
\end{proof}

\begin{imp}\label{div} If $m\in Q_M$ and $k\in\{1\sco m\}$, then $\bs{\frac{m}{k}},\bse{\frac{m}{k}}\in Q_M\cup\{1\}$.
\end{imp}

\begin{proof} Follows from Lemma~\ref{pow} and, also, Corollaries \ref{inp} and~\ref{sep}.
\end{proof}

\begin{lemma}\label{q2} If $M\ne\{0\}$, then $2\in Q_M$.
\end{lemma}

\begin{proof} By condition, $Q_M\ne\es$.

Let $m\in Q_M$ be an arbitrary element. Then, $m\ge2$, $m-1\in\{1\sco m\}$. According to Corollary~\ref{div}, $\bse{\frac{m}{m-1}}\in Q_M\cup\{1\}$. Also,
$\frac{m}{m-1}=1+\frac{1}{m-1}\in(1;2]$, $\bse{\frac{m}{m-1}}=2$.
\end{proof}

\begin{theorem}\label{pm1} If $m\in Q_M$ and $m\le\frac{n}{2}$, then at least one of two following conditions holds\:
\begin{nums}{-1}
\item\label{nei} $m+1,m-1\in Q_M\cup\{1\}$\~
\item\label{prm} $m$~is a~power of $\cha\F$.
\end{nums}
\end{theorem}

\begin{proof} By Corollary~\ref{sep}, $\{m,m\}\in\Pc_M$.

The subgroup $U_m:=\{\ep\in\F^*\cln\ep^m=1\}\subs\F^*$ is finite. Assume that it is nontrivial (otherwise, \ref{prm}~holds).

We will naturally identify the algebra $M_{2m}(\F)$ with $M_2\br{M_m(\F)}$, i.\,e. partition $\br{(2m)\bw\times(2m)}$\dh matrices into $(m\times m)$\dh blocks.
Set $J:=J_{0,m}\in M_m(\F)$.

Let $a,b\in\F^*$ be arbitrary elements. In the algebra $M_{2m}(\F)$, define the matrices
\equ{
Z:=\rbmat{J&0\\0&J};\quad\quad Z_{a,b}:=\rbmat{aJ&E\\0&bJ}.}
Clearly, $J^m=0$, $Z^m=0$, and $[Z,Z_{a,b}]=0$. Besides, for any $k\in\N$, we have
\begin{gather*}
Z_{a,b}^k=\rbmat{a^k J^k & I_{a,b}^{(k)}\\0 & b^k J^k},\\
I_{a,b}^{(k)}:=\sums{\sst{i,j\ge0;\\i+j=k-1}}\br{(aJ)^i\cdot E\cdot(bJ)^j}=\br{S_k(a,b)}\cdot J^{k-1}\in M_m(\F),\\
S_k(a,b):=\sums{\sst{i,j\ge0;\\i+j=k-1}}(a^ib^j)\in\F.
\end{gather*}
It is easy to see that $\br{S_k(a,b)}\cdot(a-b)=a^k-b^k$ ($k\in\N$) and, also,
\equ{\begin{aligned}
&Z_{a,b}^{m+1}=0\ne Z_{a,b}^{m-1};\\
&(Z_{a,b}^m=0)\quad\Lra\quad\br{S_m(a,b)=0}.
\end{aligned}}
Show that $\dep(Z_{a,b})=2$. Indeed, the subspace $\Ker(Z_{a,b})\subs\F^{2m}\cong(\F^m)^2$ consists exactly of all vectors $\rbmat{x\\y}$ such that $aJx+y=bJy=0$,
i.\,e. $y=c_1e_1$, $x=-\frac{c_1}{a}e_2+c_2e_1$ ($c_i\in\F$) and, hence, is two-dimensional. Thereafter, the matrix~$Z_{a,b}$ is nilpotent, and its Jordan form
contains exactly two cells with the set of sizes
\equ{
\bc{h(Z_{a,b}),2m-h(Z_{a,b})}=\case{
\{m,m\},&S_m(a,b)=0;\\
\{m+1,m-1\},&S_m(a,b)\ne0.}}

Let $\ep_0\in U_m\sm\{1\}$ be an arbitrary element. Then $1-\ep_0^m=0\ne1-\ep_0$ implying $S_m(1,\ep_0)\bw=0$. Hence, the Jordan form of each of the matrices $Z$
and~$Z_{1,\ep_0}$ contains two cells~$J_{0,m}$. Prove that, for some element $t\in\F$, the Jordan form of the matrix $tZ+Z_{1,\ep_0}=Z_{t+1,t+\ep_0}$
contains to cells~$J_{0,m\pm1}$. For this, it suffices that $t+1,t+\ep_0\ne0$ and $S_m(t+1,t+\ep_0)\ne0$. If $S_m(t+1,t+\ep_0)=0$, then $(t+\ep_0)^m=(t+1)^m$
and $t+\ep_0\ne t+1$ implying $t+\ep_0\bw=\ep(t+1)$ ($\ep\in U_m\sm\{1\}$), $t=\frac{\ep-\ep_0}{1-\ep}$. Since $|U_m|<\bes$, the number of elements $t\in\F$ such that
$S_m(t+1,t+\ep_0)=0$ is finite. Consequently, there exists an element $t\in\F$ satisfying $t\ne-1,-\ep_0$ and $S_m(t+1,t+\ep_0)\ne0$,\ti this element is as required.
Since $2m\le n$, there exist matrices $X,Y\in A_n$ such that $[X,Y]=0$, $G(X)=G(Y)=\{m,m\}$, and the set of the cell sizes of the Jordan form of the matrix $tX+Y\in A_n$
is $\{m+1,m-1,\und{1\sco1}_{n-2m}\}$. Also, $\{m,m\}\in\Pc_M$, $X,Y\in M$, and, thus, $M\sups\ha{X,Y}\ni tX+Y$, $G(tX+Y)\in\Pc_M$. According to Corollary~\ref{sep},
\ref{nei}~holds.
\end{proof}

\begin{theorem}\label{p2} If $m,m_1\in Q_M\cup\{1\}$, $m_1>m+2$, and $m+m_1\le n$, then $m+2\in Q_M$.
\end{theorem}

\begin{proof} Set $m_2:=m+2$. By condition, $m_1>m_2>2$, $m_2\le m_1-1$.

For any $i\in\{1;2\}$, consider an $m_i$\dh dimensional space~$V_i$, in it\ti a~basis $(e_{i,1}\sco e_{i,m_i})$ and the operator~$Z_i$ with matrix~$J_{0,m_i}$
in this basis. <<Continue>> the operators $Z_1$ and~$Z_2$ onto the space $V:=V_1\oplus V_2$ by the rule $Z_iV_j:=0$ ($i\ne j$).
Clearly, $Z_1Z_2=Z_2Z_1=0$. For the operator $Z_0:=Z_1+Z_2\in\End(V)$ and the subspaces
\equ{\begin{aligned}
&W:=\ha{e_{i,k}\cln i=1,2;\,k<m_i}\oplus\ha{e_{1,m_1}+e_{2,m_2}}\subs V;\\
&U:=\ha{e_{1,1}+e_{2,1}}\subs V,
\end{aligned}}
we have $U\subs W$, $\dim(W/U)=m_1+m_2-2=m_1+m\le n$, $Z_iU=0$, and $Z_iV\subs W$ ($0\le i\le2$). Consequently, in the space $W/U$, the operator $Z_1$ (resp.~$Z_2$)
induces an operator $\wt{Z}_1$ (resp.~$\wt{Z}_2$) satisfying $\wt{Z}_1\wt{Z}_2=\wt{Z}_2\wt{Z}_1=0$ and $\wt{Z}_0=\wt{Z}_1+\wt{Z}_2$. Thus, $[\wt{Z}_i,\wt{Z}_j]=0$
($0\le i,j\le2$) and $\wt{Z}_0^k=\wt{Z}_1^k+\wt{Z}_2^k$ ($k\in\N$).

Let $i\in\{1;2\}$ be an arbitrary number. Find the Jordan form of the operator~$\wt{Z}_i$.

We have $Z_i^{m_i}=0$ and $Z_i^{m_i-1}(e_{1,m_1}+e_{2,m_2})=e_{i,1}\notin U$ that follows $\wt{Z}_i^{m_i}=0\ne\wt{Z}_i^{m_i-1}$. Further, for the subspace
$W_i:=\ha{e_{i,k}\cln k<m_i}$, we have $Z_iW=W_i$ and $W_i\cap U=0$ implying $\rk\wt{Z}_i=\dim W_i=m_i-1$, $\dep\br{\wt{Z}_i}=\dim(W/U)-m_i+1$.
Therefore, $\wt{Z}_i$ is a~nilpotent operator whose Jordan form contains exactly $\dim(W/U)-m_i+1$ cells with the maximal size~$m_i$, i.\,e. with the
set of sizes $\{m_i,\und{1\sco1}_{\dim(W/U)-m_i}\}$.

Now, find the Jordan form of the operator~$\wt{Z}_0$. As said above, $m_2\le m_1-1$, and, hence, $\wt{Z}_1^{m_1}=\wt{Z}_2^{m_1-1}=0\ne\wt{Z}_1^{m_1-1}$,
$\wt{Z}_0^{m_1}=\wt{Z}_1^{m_1}+\wt{Z}_2^{m_1}=0$, $\wt{Z}_0^{m_1-1}=\wt{Z}_1^{m_1-1}+\wt{Z}_2^{m_1-1}\ne0$. Further,
$Z_0^{-1}(U)=\ha{e_{1,1},e_{2,1},e_{1,2}+e_{2,2}}\sups U$ implying $\dep\br{\wt{Z}_0}\le\dim\br{Z_0^{-1}(U)}-\dim U=2$. Therefore, $\wt{Z}_0$ is a~nilpotent
operator whose Jordan form contains at most two cells with the maximal size~$m_1$, i.\,e. with the set of sizes $\{m,m_1\}$.

Since $\dim(W/U)\le n$, there exist pairwise commuting matrices $X_0,X_1,X_2\bw=X_0\bw-X_1\in A_n$ with the sets of the cell sizes of Jordan forms respectively
$\{m,m_1,\und{1\sco1}_{n-(m+m_1)}\}$, $\{m_1,\und{1\sco1}_{n-m_1}\}$, and $\{m_2,\und{1\sco1}_{n-m_2}\}$. Note that $m,m_1\in Q_M\cup\{1\}$, and, by
Corollary~\ref{sep}, $G(X_{0,1})\in\Pc_M$, $X_{0,1}\in M$ implying $M\sups\ha{X_0,X_1}\ni X_2$, $G(X_2)\in\Pc_M$. On the other hand, $G(X_2)=\{m_2\}$. Hence, $m_2\in Q_M$.
\end{proof}

Suppose that $M\ne\{0\}$.

Due to Lemma~\ref{q2}, there exists a~maximal number $m_0\in\BC{2\sco\bs{\frac{n}{2}}+1}$ such that $2\sco m_0\in Q_M$. If $m_0\ne\bs{\frac{n}{2}}+1$, then
$m_0\le\bs{\frac{n}{2}}\le\frac{n}{2}$, $m_0+1\notin Q_M$, and, by Theorem~\ref{pm1}, $m_0$~is a~power of $\cha\F$. Clearly, $m':=m_0-1\in Q_M\cup\{1\}$.

Show that each number from $Q_M\sm\{2\sco m_0\}$ is at most $2m_0$ and at least $n-m_0+2$, assuming that $m_0\le\bs{\frac{n}{2}}$ and $m_0+1\notin Q_M$ (once
$m_0=\bs{\frac{n}{2}}+1$, it holds automatically as noted in~\Ss\ref{introd}).

Let $m_1\in Q_M$ be an arbitrary number greater than~$m_0$. We have $m'+2=m_0+1\notin Q_M$ implying $m_1>m_0+1=m'+2$. Besides, $m',m_1\in Q_M\cup\{1\}$, and, according
to Theorem~\ref{p2}, $m'+m_1>n$, $m_1+m_0=m_1+m'+1>n+1$, $m_1+m_0\ge n+2$. Hence, $m_1\ge n-m_0+2$ and, also, $2m_1>m_1+m_0\ge n+2$, $m_1>\frac{n}{2}+1$.

Suppose that there exists a~number $m\in Q_M$ greater than $2m_0$. By Corollary~\ref{div}, $m_1:=\bse{\frac{m}{2}}\in Q_M\cup\{1\}$. Also, $m_1\ge\frac{m}{2}>m_0>1$
implying $m_1\in Q_M$. Consequently, $m_1>\frac{n}{2}+1$. On the other hand, $m_1\le\bse{\frac{n}{2}}<\frac{n}{2}+1$. So, we came to a~contradiction.

Thus, all numbers from $Q_M\sm\{2\sco m_0\}$ are at most $2m_0$ and at least $n-m_0+2$.

So, we proved the <<only if>> statement in Theorem~\ref{main}. Let us prove the <<if>> one.

Suppose that a~subset $Q\subs\{2\sco n\}$ and a~number $m_0\in\BC{2\sco\bs{\frac{n}{2}}+1}$ equal either to $\bs{\frac{n}{2}}+1$ or to $(\cha\F)^k$ \ter{$k\in\N$}
satisfy \eqref{fir} and~\eqref{oth}. Show that $M(Q)\in\Mc_n$.

The inequality $\bs{\frac{n}{2}}+1>\frac{n}{2}$, Corollary~\ref{mhal}, and Lemma~\ref{mdeg} imply that, if $X_1,X_2\in A_n$ and $X_1^{m_0}=X_2^{m_0}=[X_1,X_2]=0$,
then $Y^{m_0}=0$ for any $Y\in\ha{X_1,X_2}$.

Let $X_1,X_2\in M(Q)$ be arbitrary commuting matrices. In view of their nilpotence, $\ha{X_1,X_2}\subs A_n$. It is required to prove that $\ha{X_1,X_2}\subs M(Q)$.
Assume that there exists a~matrix $Y\in\ha{X_1,X_2}\subs A_n$ not belonging to $M(Q)$. Its Jordan form contains a~cell of size $l\notin Q\cup\{1\}$, and, hence, $l>m_0$,
$Y^{m_0}\ne0$. Consequently, there exists a~number $i\in\{1,2\}$ such that $X_i^{m_0}\ne0$, $m:=h(X_i)>m_0$. Since $X_i\in M(Q)$, we have $m=h(X_i)\in Q\cup\{1\}$,
$m\ne l$. Further, $m>m_0>1$ implying $m\in Q$, $n-m_0+2\le m\le2m_0$, $2m>m+m_0\ge n+2$, $m>\frac{n}{2}+1$. Therefore, $m_0\ge n-m+2$ and
$\bse{\frac{m}{2}}\le\bse{\frac{2m_0}{2}}=m_0$. Thus, $m>\frac{n}{2}+1$, $l>m_0\ge n-m+2,\bse{\frac{m}{2}}$, and $l\ne m$. Finally, $[X_i,Y]=0$. It contradicts
Corollary~\ref{hx} and shows that $\ha{X_1,X_2}\subs M(Q)$.

So, we completely proved Theorem~\ref{main} and, consequently, Theorem~\ref{submain}.

\section*{Acknowledgements}

The author is grateful to Prof. \fbox{E.\,B.\?Vinberg} for exciting an interest to fundamental algebraic science.

The author dedicates the article to E.\,N.\?Troshina.

\newpage

{\renewcommand{\refname}{References}
}

\end{document}